\documentclass[11pt,a4paper, envcountsame,mathserif]{amsart}
\usepackage[usenames,dvipsnames]{color}

\usepackage[utf8]{inputenc}

 \usepackage[colorlinks,citecolor=blue,urlcolor=blue, linkcolor=blue, backref]{hyperref}

\usepackage[capitalise]{cleveref}		

 \usepackage{amsmath, amssymb, xspace}
 \usepackage{graphicx}

 \usepackage[all]{xy}

\usepackage{tikz}
\usetikzlibrary{arrows,snakes,positioning,backgrounds,shadows}
\usepackage{stmaryrd}

\newtheorem{theorem}{Theorem}[section]

\newtheorem{thm}[theorem]{Theorem}

\newtheorem{fact}[theorem]{Fact}

\newtheorem{proposition}[theorem]{Proposition}

\newtheorem{prop}[theorem]{Proposition}

\newtheorem{claim}[theorem]{Claim}

\newtheorem{conjecture}[theorem]{Conjecture}

\newtheorem{lemma}[theorem]{Lemma}		

\newtheorem{corollary}[theorem]{Corollary}

\newtheorem{cor}[theorem]{Corollary}

\newtheorem{question}[theorem]{Question}

\theoremstyle{definition}
\newtheorem{definition}[theorem]{Definition}

\newtheorem{remark}[theorem]{Remark}

\newcommand{\NN}{{\mathbb{N}}}

\newcommand{\QQ}{{\mathbb{Q}}}
\newcommand{\ZZ}{{\mathbb{Z}}}
\newcommand{\sub}{\subseteq}
\newcommand{\sN}[1]{_{#1\in \NN}}
\newcommand{\uhr}[1]{\! \upharpoonright_{#1}}

\newcommand{\SI}[1]{\Sigma^0_{#1}}
\newcommand{\PI}[1]{\Pi^0_{#1}}

\newcommand{\bi}{\begin{itemize}}
\newcommand{\ei}{\end{itemize}}
\newcommand{\bc}{\begin{center}}
\newcommand{\ec}{\end{center}}

\newcommand{\tp}[1]{2^{#1}}
\newcommand{\ex}{\exists}
\newcommand{\fa}{\forall}

\newcommand{\la}{\langle}
\newcommand{\ra}{\rangle}

\newcommand{\n}{\noindent}

\newcommand{\vsp}{\vspace{6pt}}

\newcommand{\sss}{\sigma}

\newcommand{\w}{\omega}

\newcommand \seq[1]{{\left\langle{#1}\right\rangle}}

\newcommand\+[1]{\mathcal{#1}}

\newcommand{\ol}{\overline}

\newcommand{\lra}{\leftrightarrow}
\newcommand{\LR}{\Leftrightarrow}

\newcommand{\DA}{\downarrow}

\newcommand{\dom}{\ensuremath{\mathrm{dom}}}

\newcommand{\dset}[2]{\{#1 : #2 \}}

\DeclareMathOperator{\gra}{graph}

\newenvironment{claimproof}{\begin{proof}}{\end{proof}}

\newcommand{\frb}{\mathfrak{b}}

\numberwithin{equation}{section}
\renewcommand{\hat}{\widehat}
\begin{document}

\title{Logic Blog 2017}

 \author{Editor: Andr\'e Nies}

\maketitle


 {
The Logic   Blog is a shared platform for
\bi \item rapidly announcing  results and questions related to logic
\item putting up results and their proofs for further research
\item parking results for later use
\item getting feedback before submission to  a journal
\item foster collaboration.   \ei

Each year's   blog is    posted on arXiv shortly after the year has ended.
\vsp
\begin{tabbing} 

 \href{http://arxiv.org/abs/1703.01573}{Logic Blog 2016} \ \ \ \   \= (Link: \texttt{http://arxiv.org/abs/1703.01573})  \\
 
  \href{http://arxiv.org/abs/1602.04432}{Logic Blog 2015} \ \ \ \   \= (Link: \texttt{http://arxiv.org/abs/1602.04432})  \\
  
  \href{http://arxiv.org/abs/1504.08163}{Logic Blog 2014} \ \ \ \   \= (Link: \texttt{http://arxiv.org/abs/1504.08163})  \\

   \href{http://arxiv.org/abs/1403.5719}{Logic Blog 2013} \ \ \ \   \= (Link: \texttt{http://arxiv.org/abs/1403.5719})  \\

    \href{http://arxiv.org/abs/1302.3686}{Logic Blog 2012}  \> (Link: \texttt{http://arxiv.org/abs/1302.3686})   \\

 \href{http://arxiv.org/abs/1403.5721}{Logic Blog 2011}   \> (Link: \texttt{http://arxiv.org/abs/1403.5721})   \\

 \href{http://dx.doi.org/2292/9821}{Logic Blog 2010}   \> (Link: \texttt{http://dx.doi.org/2292/9821})  
     \end{tabbing}

\vsp

\n {\bf How does the Logic Blog work?}

\vsp

\n {\bf Writing and editing.}  The source files are in a shared dropbox.
 Ask Andr\'e (\email{andre@cs.auckland.ac.nz})  in order    to gain access.

\vsp

\n {\bf Citing.}  Postings can be cited.  An example of a citation is:

\vsp

\n  H.\ Towsner, \emph{Computability of Ergodic Convergence}. In  Andr\'e Nies (editor),  Logic Blog, 2012, Part 1, Section 1, available at
\url{http://arxiv.org/abs/1302.3686}.}

\vsp 

%
%
%
 
The logic blog,  once it is on  arXiv,  produces citations on Google Scholar.
%
\newpage
\tableofcontents

\part{Computability theory}
   
       \newtheorem{oq}{Open Question}

\section{Open questions from Capulalpan retreat, December 2016}

Eight researchers met on a four-day retreat in Mexico. Here is a collection of open questions that were discussed.
\subsection{Jason Rute and Rutger Kuyper}
\begin{thm}[Miller and Kuyper]
$A \in 2^\omega$ is K-trivial iff 

$\forall \in 2^\omega\ [X \text{ is MLR } \rightarrow X \triangle A \text{ is MLR}]$.
\end{thm}

Kuyper and Rute think they can generalize this theorem to compact groups---$\triangle$ is replaced with a group operation, MLR is for the Haar measure, and K-trivial is as in the Melnikov and Nies paper \cite{Melnikov.Nies:12}.  However, the following similar questions remain open.

In this next question, $X / A$ means the bits of $X$ selected using $A$.  A significant difference here is that selection is not invertable and therefore not a group operation or even a group action.
\begin{oq}
Is the following true?  A  set $A \in 2^\omega$ is K-trivial iff $\forall  X  \in 2^\omega\ [X \text{ is MLR } \rightarrow X /  A \text{ is MLR}]$.
\end{oq}

In this next question, $E(2)$ is the (non-Abelian) group of orientation preserving transformations of $\mathbb{R}^2$.  $E(2)$ is entirely composed of transformations $T$ where $T$ is a rotation followed by a translation, so $E(2)$ can be thought of as the space $\mathbb{T}\times\mathbb{R}^2$ (where $\mathbb{T}$ is the circle) with the appropriate group action.  (For a similar question one could consider the compact group $SO(2)$, the orientation preserving transformations of the sphere.)  $E(2)$ acts on $\mathbb{R}^2$, and the Lebesgue measure is the unique locally finite invariant Borel measure (up to scaling).
\begin{oq}
Is the following true?  A transformation $A \in E(2)$ is K-trivial iff $\forall (x,y) \in 2^\omega\ [(x,y) \text{ is MLR } \rightarrow A(x,y) \text{ is MLR}]$.
\end{oq}

Since this group is a merely \emph{acting on} $\mathbb{R}^2$ it doesn't seem to instantly follow from Kuyper and Rute's result mentioned above.  In particular, $E(2)$ has an extra dimension, so there are continuum-many ways to send $(x_1,y_1)$ to $(x_2,y_2)$.  This may or may not change the answer.

There are many more questions of this type.  Here is one more.
\begin{oq}
Is the following true?  A real $r \in (0,\infty)$ is K-trivial iff for all MLR Brownian motions $B$, we have that $B(r)$ is MLR (for the Lebesgue measure).
\end{oq}

\subsection{Brown Westrick}
Let $\mathcal P \subseteq 2^{\mathbb Z}$ be a subshift. Let $D (\mathcal P)$ be the set of effective Hausdorff dimensions of members of $\mathcal P$. By Simpson paper, $D (\mathcal P) \subseteq [0,h(\mathcal P)]$ where $h$ denotes the topological entropy of $\mathcal P$ (inf over $n$ of log of number of $n$-patterns occurring divided by $n$). If $\mathcal P$ is of finite type then we have equality. 

Say that $x $ is $d,b$-shift-complex if $K(\sigma) \ge d |\sigma| -b$ for each pattern $\sigma$ that occurs in $x$ (see BSL 2013 survey by Khan). Say that  $\mathcal  P$ is $d,b$-shift complex  if each $x \in \mathcal P$ is $d,b$-shift complex. In this case   $d$ is a lower bound for $D(\mathcal P)$.  
\begin{oq} Is $D(\mathcal P)$ necessarily closed? \end{oq} 
\begin{oq} Is some  $x \in \mathcal P$ necessarily $h(\mathcal P), O(1)$ shift complex? \end{oq}
\subsection{Denis Hirschfeldt}
For definitions see below.
\begin{oq}
Are all $1$-random sets quasiminimal in the uniform dense degrees?
\end{oq}

\begin{oq}
Are there minimal pairs in the [uniform or nonuniform]  [generic or
effective dense] degrees?
\end{oq}
Related results on the four other reducibilities below by Igusa; Hirschfeldt, Jockusch, Kuyper and Schupp; Cholak and Igusa (contains some work joint with Hirschfeldt); Astor, Hirschfeldt and Jockusch (in preparation).

\begin{definition}
Let $g : \omega \to \omega$. A {\bf partial description} of $g$ is a
partial function $f : \omega \to \omega$ such that $f(n)=g(n)$
whenever $f(n)$ is defined. A {\bf generic description} of $g$ is a
partial description of $g$ with domain of density $1$.

A {\bf dense description} of a function $g : \omega \to \omega$ is
a partial function $f : \omega \to \omega$ such that
$f(n)\DA=g(n)$ on a set of density $1$.

For a function $f : \omega \to \omega \cup \{\square\}$, the
{\bf strong domain} of $f$ is $f^{-1}(\omega)$.  Let $g : \omega \to
\omega$. A {\bf strong partial description} of $g$ is a (total)
function $f : \omega \to \omega \cup \{\square\}$ such that
$f(n)=g(n)$ on the strong domain of $f$. An {\bf effective dense
description} of $g$ is a strong partial description of $f$ with
strong domain of density $1$.
\begin{itemize} \item
We say that $h$ is \emph{nonuniformly generically reducible} to $g$,
and write $h \leq_{ng} g$, if for every generic description $f$ of
$g$, there is an enumeration operator $W$ such that $W^{\gra(f)}$
enumerates the graph of a generic description of $h$.

\item We say that $h$ is \emph{uniformly generically reducible} to $g$, and
write $h \leq_{ug} g$, if there is an enumeration operator $W$ such
that if $f$ is a generic description of $g$, then $W^{\gra(f)}$ is a
generic description of $h$.

\item We say that $h$ is \emph{nonuniformly densely reducible} to $g$, and
write $h \leq_{nd} g$, if for every dense description $f$ of $g$,
there is an enumeration operator $W$ such that $W^{\gra(f)}$
enumerates the graph of a dense description of $h$.

\item We say that $h$ is \emph{uniformly densely reducible} to $g$, and
write $h \leq_{ud} g$, if there is an enumeration operator $W$ such
that if $f$ is a dense description of $g$, then $W^{\gra(f)}$ is a
dense description of $h$.

\item We say that $h$ is \emph{nonuniformly effectively densely reducible}
to $g$, and write $h \leq_{ned} g$, if every effective dense
description of $g$ computes an effective dense description of
$h$.

\item We say that $h$ is \emph{uniformly effectively densely reducible} to
$g$, and write $h \leq_{ued} g$, if there is a Turing functional
$\Phi$ such that if $f$ is an effective dense description of $g$, then
$\Phi^f$ is an effective dense description of $h$.
\end{itemize}
Let $$\mathcal R(A)=\{2^nk : n \in A \, \wedge \, k \textrm{ odd}\}.$$
Let $J_n=[2^n,2^{n+1})$ and let $$\widetilde{\mathcal R}(A) =
\bigcup_{n \in A} J_n.$$ Let $\mathcal E(A) = \widetilde{\mathcal
R}(\mathcal R(A))$. This operator induces embeddings of the Turing
degrees into all of the degree structures arising from the above
reducibilities. In any of these structures, a degree is
\emph{quasiminimal} if it is not above any nontrivial degree in the
image of the embedding induced by $\mathcal E$.
\end{definition}

\subsection{Andre Nies}

For $K$-trivial sets $A,B$ we say that $A \le_{ML} B $ if every ML-random oracle $Z$ computing $B$ computes $A$.

\begin{oq} Is $\le_{ML}$ arithmetical? \end{oq}

Note that by Gandy basis theorem, if $A \not \le_{ML} B$ then there is a witness  $Z \le_T \mathcal O$.

A $K$-trivial set $A$  is called smart if every ML-random $Z \ge_T A$ computes all the $K$-trivials.  Equivalently, $A$ is ML-complete for the $K$-trivials. It is not even clear whether smartness is arithmetical. See Section 3 on the Logic Blog 2016 for detail. 
\begin{oq}  Can a smart $K$-trivial be cappable? Is there a Turing minimal pair of smart $K$-trivials? \end{oq}

\begin{oq} Suppose $A$ is $K$-trivial. Is $A$ Turing below each LR-hard ML-random? \end{oq}

\begin{oq} Is weak 2-randomness closed upward under $\le_K$? (Miller and Yu). \end{oq}

\begin{oq} Is weak 2-randomness closed downward within the 1-randoms under $\le_{LR}$?
\end{oq}

One could also try to show that $e+ \pi \not \in \QQ$. Or that $e \pi \not \in \QQ$. Good news: at least one of them holds.  	
        \section{Khan, Nies: SNR functions versus DNR functions}

We study three closely related mass problems, and also their variants where a computable growth  bound is imposed on the functions. 
\newcommand{\SNR}{\mathsf{SNR}}
\newcommand{\SNPR}{\mathsf{SNPR}}
\newcommand{\DNR}{\mathsf{DNR}}

		\begin{definition}
 \ 		\bi \item[(i)] 	A function $f: \omega \rightarrow \omega$ is \emph{strongly nonrecursive} (or \emph{$\SNR$}) if for every recursive function $g$, for all but finitely many $n \in \omega$, $f(n) \neq g(n)$.
			
			 \item[(ii)] A function $f$  is \emph{strongly non-partial-recursive} (or \emph{$\SNPR$}) if for every partial recursive function $\psi$, for all but finitely many $n$, if $\psi(n)$ is defined, then $f(n) \neq \psi(n)$.
			 
			\item[(iii)] 		 A function $f$  is \emph{diagonally  non recursive} (or \emph{$\DNR$}) if     $f(n) \neq J(n)$ whenever $J(n)$ is defined. Here $J$ is a fixed universal p.r.\  function, e.g.\ $J(n) \simeq \phi_n(n)$, though below we will  use a different one. \ei
		\end{definition}

Trivally $\SNPR$ implies $\SNR$. Also, if $f$ is $\SNPR$ then a finite variant of $f$ is $\DNR$. $\SNR$ has an analog in cardinal characteristics called $\frb(\neq^*)$ (\cite[Section 6]{LogicBlog:16}). Anything in computability involving enumeration/partiality fails to have such an analog.

We will show that every non-high $\SNR$ function is $\SNPR$ and hence computes a $\DNR$. Also,   every $\DNR$ function    computes an $\SNPR$ function. We can also keep track of bounds  on the functions that are  order functions (OF) as defined below. 
   Theorems 3.8 and 3.10 in \cite{Khan.Miller:17} yield a downward and an upward growth hierarchy within $\DNR$: for every OF $g$, there is a (much faster growing) OF $h$ such that there is an $h$-bounded $\DNR$ function that computes no $g$-bounded $\DNR$ function. A similar result holds with $g$ and $h$ interchanged. Our  translation  between  $\SNR$ and $\DNR$ can be used to obtain similar hierarchy results for $\SNR$.

If $A$ is high then it computes a function $f$ dominating all computable functions, which is in particular $\SNR$. On the other hand not each high set $A$ computes a $\DNR$ function (e.g., a high incomplete r.e.\ set $A$ doesn't). We discuss that outside the high degrees, the degree classes of such functions are the same. We also check how potential   computable bounds  on the functions change when going from one class to the other. The facts  suggests that $\SNR$ for the same bound  is stronger. However, going from $\DNR$ to $\SNR$ the loss is still within the elementary.

 By the following, highness  is the only reason an $\SNR$ function  can fail to compute a $\DNR$ function. The result is due to  Kjos-Hansen, Merkle, and Stephan \cite[Thm.\ 5.1 $(1) \to (2)$]{Kjos.Merkle:11}
\begin{prop}
			Every non-high $\SNR$ function is $\SNPR$ and hence computes a $\DNR$.
		\end{prop}
		\begin{proof}
			Supose that $f:\omega \rightarrow \omega$ is not high, and that $\psi$ is a partial recursive function that is infinitely often equal to it. For each $n \in \omega$, let $g(n)$ be the least stage such that $|\dset{x \in \omega}{\psi(x)[g(n)]\DA  = f(x)}| \ge 2n $. Then $g$ is recursive in $f$.

			Since $f$ is not high, there is a recursive function $h$ that escapes $g$ infinitely often. We define a recursive function $j$ that is infinitely often equal to $f$. Let $j_0 = \emptyset$. Given $j_n$, let \[A = \dset{\langle x, \psi(x)\rangle}{x \notin \dom(j_n), \psi(x)[h(n)] \DA}.\]

			Let $y$ be the least such that it is not in the domain of $j_n \cup A$. Finally, let $j_{n+1} = j_n \cup A \cup \{\langle y, 0 \rangle\}$. 

			Clearly  $j = \bigcup_n j_n$ is recursive. To see that $j(x) = f(x)$ for   infinitely many $x$, take $n$ such that $h(n) > g(n)$. Then there are $2n$ many $x$ such that  we have a coincidence $f(x) = \psi(x)[h(n)]$. We have lost at most $n$ coincidences by defining $j_{k+1}(y)=0$ at stages $k<n$. Thus $j_{n+1}(x) = f(x)$ for at least $n$ many $x$.
		\end{proof}

	\begin{definition}
			An \emph{order function} is a recursive, nondecreasing, and unbounded function $p: \omega \rightarrow \omega$ such that $p(0) \ge 2$.  \end{definition}
			\begin{definition}
			For a class $\mathsf{C}$ of functions from $\omega$ to $\omega$ and an order function $p$, let $\mathsf C_p$ denote the subclass consisting of those functions $f$ such that $f(n) < p(n)$ for each $n$.
		\end{definition}

In the following we define the universal p.r.\ functional by $$J(2^e (2x+1)) \simeq \phi_e(x).$$ 		
	\begin{prop} Every $\DNR$ function $g$ (for $J$ as above)  computes an $\SNPR$ function $h$. The reduction is fixed, i.e.\ $\DNR \ge_S \SNR$ (Medvedev). Furthermore, if $g \in \DNR_p$ then we can arrange $h\in \SNPR_q$ where $q(n) = \prod_{i \le n \log n} p(C i)$ for some constant $C$. \end{prop}
\begin{proof} We modify the argument in the proof of \cite[Thm.\ 7]{Jockusch:89}
Given a fixed effective encoding of tuples of natural numbers by natural numbers, we let $(n)_u$ be the $u$-th entry of the tuple coded by $n$, if any, and vacuously $(n)_u=0$ otherwise.  Let $r$ be a computable function such that $J(r(u)) \simeq J(u)_u$ for each $u$. (Thus, $r(u)= 2^j (2u+1)$ where $i$ is an index such that $\phi_j(u) \simeq J(u)_u$.)

Let $d$ be a computable function such that $ n = o(d(n))$, e.g.\ $d(n) = n \log n$. Now let $h$ be a computable function such that  \[ \fa u \le d(e)  \, h(e)_u = f(r(u)). \]
That is, $h(e)$ encodes the initial segment of values of the function $g \circ r$ up to length $d(e)$.

Since $g$ is $\DNR$, for each $u \le d(e)$ we have 
\bc $h(e)_u = g(r(u)) \neq J(r(u)) = J(u)_u$. \ec
In particular, for $d(e) \ge u= 2^i (2e+1)$ we have $h(e) \neq J(u) \simeq \phi_i(e)$. Thus $h$ is $\SNPR$.

Suppose $f \in \DNR_p$. We can choose the encoding of initial segments $g \circ r \uhr {d(e)}$ via numbers of size bounded by $q$ (with $C = 2^{j+2}$, $j$ the index given above). Thus we can ensure $h<q$.
\end{proof}

\begin{question} Is there an order function $p$ and a $\DNR_p$ that does not compute an $\SNR_p$? \end{question}

\section{Khan, Beros, Nies, Kjos-Hanssen: \\ potential weakening  of effective (bi-)immunity}
The researchers above discussed the following during Nies' visit at UHM in October 2016. 

$A \sub \NN$ is immune if it contains no infinite c.e.\ set. Starting from Post, and then Jockusch and others, people studied an  effective version of this: $A$ is effectively immune   (e.i.) if there is a computable function $h$ such that $W_e \sub A \to |W_e| \le h(e)$.  
Also $A$ is effectively bi-immune if $A, \NN - A$ are e.i. There is a lot of work comparing the degrees of such sets with the degrees of d.n.c.\ functions: Jockusch 1989 show that these degrees coincide with the degrees of e.i.\ sets, and Lewis and Jockusch 2013 that every d.n.c. computes a bi-immune. Beros showed that not every d.n.c.\ computes an e.b.i.

Now let $\seq{R_e}$ be a listing of the computable sets, say $R_e$ is the ascending part of $W_e$, only admitting an element if it is greater than the previously enumerated ones. As every infinite c.e.\ set has an infinite computable subset, immunity doesn't change when we restrict to computable instead of c.e.\ subsets.  This may be different for the effective versions,   which could be weaker in the sense of Muchnik reducibility. Define  computably e.i, computably e.b.i.\ as above but using the listing $\seq {R_e}$.

\begin{oq}  \ \bi\item[(i)] Does  every computably e.i.\ set compute an e.i.\ set? \item[(ii)] Does every computably e.b.i.\  set compute an e.b.i.\ set? \ei \end{oq}
 \part{Higher computability theory/effective descriptive set theory}
 \section{Yu: $\Delta^1_2$-degree determinacy}
This is joint work with CT Chong and Liuzhen Wu.

It is obvious (by Mansfield-Solovay's argument) that if $A$ is $\Sigma^1_2$ and not thin, then $A$ ranges over an upper cone of $L$-degrees. Now it was asked by some people in the Dagstuhl workshop end of February whether it can range over an upper cone of $\Delta^1_2$-degrees. The subtle thing is that if one does a Cantor-Bendixson derivation over a Suslin  representation of non-thin $\Sigma^1_2$ set,  it may go through $(\omega_1)^L$ which is bigger than $\delta^1_2$, the least ordinal which cannot be represented by a $\Delta^1_2$ well ordering over $\omega$. However, the answer is still yes.

\begin{proposition}\label{proposition: delta12 degree determinacy}
Suppose that there is a nonconstructible real. If $A$ is $\Sigma^1_2$ and not thin, then $A$ ranges over an upper cone of $\Delta^1_2$-degrees.  Actually there is a $\Delta^1_2$-coded perfect set $S\subseteq A$.
\end{proposition}
\begin{proof}
Since $A$ is $\Sigma^1_2$, there is a $\Pi^1_1$ set $B\subseteq (\omega^{\omega})^2$ so that $\forall x (x\in A\leftrightarrow \exists y((x,y)\in B))$. By $\Pi^1_1$-uniformization, we may assume that  $\forall x(\exists y(x,y)\in B\rightarrow \exists ! y(x,y)\in B)$. If

By Slaman's result, $A$ must contain a nonconstructible real.  Then so is $B$ and so $B$ is not thin. Now let \begin{multline*} C=\{T\subseteq (\omega^{<\omega})^2\mid [T]\subseteq B \wedge \forall (\sigma,\tau)\in T  \exists (\sigma_0,\tau_0)\in T\exists (\sigma_1,\tau_1)\in T \\ \sigma_0\succ\sigma\wedge\sigma_1\succ \sigma\wedge \sigma_0|\sigma_1\}.\end{multline*} Then $C$ is a $\Pi^1_1$ nonempty set and so must contain an element $T\in \Delta^1_2$. Now it is easy to construct a $\Delta^1_2$-coded perfect set $S\subseteq A$ from $T$.
\end{proof}



If the assumption of Proposition \ref{proposition: delta12 degree determinacy} is dropped, then the first part still holds. Also note that the argument in Proposition \ref{proposition: delta12 degree determinacy} does not work if the assumption is dropped. For example, $x\in L\cap \mathbb{R}$ if and only if there is a real $y\in L_{\omega_1^y}$ so that $x\leq_T y$. Then let $B=\{(x,y)\mid x\leq_T y \wedge y\in L_{\omega_1^y}\}$ be a $\Pi^1_1$-thin set. We have that $x\in L\cap \mathbb{R}\leftrightarrow \exists y (x,y)\in B$. 

\begin{proposition}\label{proposition: delta12 delta12}
If $A$ is a ZFC-provable $\Delta^1_2$ non-thin set, then $A$ contains a $\Delta^1_2$-perfect subset.\end{proposition}
\begin{proof}
By Proposition \ref{proposition: delta12 degree determinacy}, it is sufficient to assume that every real is constructible. So there is a perfect tree $T\in L$ so that $T\subseteq A$. Adding a Cohen $g$ real to $V$, then by Shoenfield absoluteness, $V[g]\models T\subseteq A$ since $A$ is $\Delta^1_2$. Then $V[g]\models A \mbox{ contains a perfect subset}$. By Proposition \ref{proposition: delta12 degree determinacy}, $V[g]\models A \mbox{ contains a }\Delta^1_2 \mbox{-perfect subset } \tilde{T}$. So $\tilde{T}\in V$. By Shoenfield absoluteness again, $V \models A \mbox{ contains a }\Delta^1_2 \mbox{-perfect subset } \tilde{T}$.
\end{proof}

\begin{lemma}
If every real is constructible, then there is a co-countable $\Delta^1_2$-set $A$ having no $\Delta^1_2$-perfect subset.
\end{lemma}
\begin{proof}
We $L_{\omega_1}$-recursively build a set $A$ and $B$ such that $A=2^{\omega}\setminus B$ as a follows:

Fix an $L_{\omega_1}$-effective enumeration of $\Delta^1_1$-perfect trees $\{T_{\beta}\}_{\beta< \omega_1}$ (of course there are at most countably many such trees, but $L$ dose not know this without using parameters).

At stage $\gamma<\omega_1$, if $\bigcup_{\gamma'<\gamma}B_{\gamma'}\cap [T_{\gamma}]$ is not empty, then let $A_{\gamma}=\bigcup_{\gamma'<\gamma}A_{\gamma'}$,  $B_{\gamma}=\bigcup_{\gamma'<\gamma}B_{\gamma'}$, and go to next stage. Otherwise, pick up $<_L$-least real $x\in [T_{\gamma}]\setminus L_{\gamma}$ and let $B_{\gamma}=\bigcup_{\gamma'<\gamma}B_{\gamma'}\cup \{x\}$. Define $A_{\gamma}=(L_{\gamma}\cap 2^{\omega})\setminus B_{\gamma}$.

Then both $A=\bigcup_{\gamma<\omega_1}A_{\gamma}$ and $B=\bigcup_{\gamma<\omega_1}A_{\gamma}$ are r.e. in $L_{\omega_1}$ and $A=2^{\omega}\setminus B$. 

So $A$ is $\Delta^1_2$. By the construction, $A$ contains no $\Delta^1_2$-perfect subset.
\end{proof}

So we have the following result.
\begin{theorem}
Every $\Sigma^1_2$ nonthin set has a $\Delta^1_2$-perfect subset if and only if there is a non-constructible real.
\end{theorem}

 \part{Randomness, analysis and ergodic theory}
  \section{Turetsky: adaptive cost functions and the Main Lemma}

\newcommand{\cost}{\mathbf{c}}
\newcommand{\dost}{\mathbf{d}}

The Main Lemma related to the Golden Run method \cite[5.5.1]{Nies:book} states the following:

\begin{lemma}
If $M$ is a prefix-free oracle machine, and $A$ is a $K$-trivial set with some computable approximation $\seq{A_s}$, then there is a computable sequence $q_0 < q_1 < \dots$ with
\[
\sum_s \sum_\rho 2^{-|\rho|} \left\llbracket M^{A}(\rho)[q_s]\DA \& \ m_s < \text{use}\, M^{A}(\rho)[q_s] \le q_{s-1}\right\rrbracket < \infty,
\]
where $m_s$ is least with $A_{q_s}(m_s) \neq A_{q_{s+1}}$.
\end{lemma}

I understand this as a statement that $K$-trivial sets obey subadditive adaptive cost functions.

\begin{definition}
An {\em adaptive cost function} is a functional $\cost^X(n,s)$ such that:
\begin{itemize}
\item For every $X$, $\cost^X(n,s)$ is total;
\item For every $X$, $\cost^X(n,s)$ is a cost function (monotonic and with limit condition);
\item For every $X$, $n$ and $s$, the use of $\cost^X(n,s)$ is $s$.
\end{itemize}
As usual, we assume that $\cost^X(n,s) = 0$ for $n \ge s$.

If in addition $\cost^X(n,s)$ is (sub)additive for every $X$, then $\cost$ is a {\em (sub)additive} adaptive cost function.
\end{definition}

\begin{lemma}
Every subadditive adaptive cost function is bounded by an additive adaptive cost function.
\end{lemma}

\begin{proof}
Suppose $\cost^X(n,s)$ is a subadditive adaptive cost function, meaning $\cost^X(n,s) \ge \cost^X(n,m) + \cost^X(m,s)$ for every $X$ and $n < m < s$.

Define $\dost^X(n,s) = \cost^X(0,s) - \cost^X(0,n)$.  Then $\dost$ is clearly an additive adaptive cost function.  Further, $\cost^X(0,s) - \cost^X(0,n) \ge \cost^X(0,n) + \cost^X(n,s) - \cost^X(0,n) = \cost^X(n,s)$ by subadditivity, and so $\dost$ bounds $\cost$.
\end{proof}

\begin{definition}
If $\cost$ is an adaptive cost function, and $\seq{A_s}$ is a computable approximation to a $\Delta^0_2$ set, we say that {\em $\seq{A_s}$ obeys $\cost$}, written $\seq{A_s} \models \cost$, if 
\[
\sum_s \cost^{A_s}(n_s,s) < \infty,
\]
where $n_s$ is least with $A_s(n_s) \neq A_{s+1}(n_s)$.

We say that a $\Delta^0_2$ set {\em $A$ obeys $\cost$}, written $A \models \cost$, if there is a computable approximation to $A$ which obeys $\cost$.
\end{definition}

\begin{prop}
If $A$ is $K$-trivial and $\cost$ is a subadditive adaptive cost function, then $A \models \cost$.
\end{prop}

\begin{proof}
It suffices to treat the case that $\cost$ is additive.  Fix a computable approximation $\seq{A_s}$ to $A$.  We must construct an appropriate prefix-free oracle machine and apply the main lemma.

We may assume that $\cost^X(s-1,s)$ is always a dyadic rational.  By the proof of the machine existence theorem, there is a computable sequence of finite sets $B_\sigma \subset 2^{<\w}$ indexed by $\sigma \in 2^{<\w}$ such that:
\begin{itemize}
\item For every $\sigma \subset \tau$, $B_\sigma \cup B_\tau$ is an anti-chain;
\item $\sum_{\rho \in B_\sigma} 2^{-|\rho|} = \cost^{\sigma}(|\sigma-1|,|\sigma|)$.
\end{itemize}
We define $M^\sigma_{|\sigma|}(\rho)\DA$ for $\rho \in B_\sigma$.  By the first property above, the domain of $M^X$ is an anti-chain for every $X$, and so $M$ is prefix-free.  Note that $M^\sigma_{|\sigma|} = M^\sigma_s$ for every $s > |\sigma|$.

Now, let $q_0 < q_1 < \dots$ be as guaranteed by the main lemma.  By pruning the first few terms if necessary, we may assume that $q_0 \ge 2$, and so in general $q_s \ge s+2$.  We claim that $\seq{D_s} \models \cost$, where $D_s = A_{q_s}$.  For if $n_s$ is least with $D_s(n_s) \neq D_{s+1}(n_s)$, and letting $\sigma_k = D_s\uhr{k+1}$, then
\begin{eqnarray*}
\cost^{D_s}(n_s,s) &=& \sum_{k = n_s}^{s-1} \cost^{\sigma_k}(k,k+1)\\
&=& \sum_{k = n_s}^{s-1} \, \sum_{\rho \in B_{\sigma_k}} 2^{-|\rho|}\\
&=& \sum_{k = n_s}^{s-1} \, \sum_\rho 2^{-|\rho|} \left\llbracket \rho \in \dom(M^{\sigma_k}_{|\sigma_k|}) \ \& \ \rho \not \in \dom(M^{\sigma_{k-1}}_{|\sigma_k|-1}) \right\rrbracket\\
&=& \sum_\rho 2^{-|\rho|} \left\llbracket \rho \in \dom(M^{D_s\uhr{s+1}}_s) \ \& \ \rho \not \in \dom(M^{D_s\uhr{n_s}}_{n_s})\right\rrbracket\\
&=& \sum_\rho 2^{-|\rho|} \left\llbracket \rho \in \dom(M^{D_s\uhr{s+1}}_s) \ \& \ \rho \not \in \dom(M^{D_s\uhr{n_s}}_{s})\right\rrbracket\\
&=& \sum_\rho 2^{-|\rho|} \left\llbracket M^{D_s}_s(\rho)\DA \& \ n_s < \text{use}\, M^{D_s}_s(\rho) \le s+1\right\rrbracket\\
&=& \sum_\rho 2^{-|\rho|} \left\llbracket M^{D_s}_{q_s}(\rho)\DA \& \ n_s < \text{use}\, M^{D_s}_{q_s}(\rho) \le s+1\right\rrbracket\\
&\le& \sum_\rho 2^{-|\rho|} \left\llbracket M^{D_s}_{q_s}(\rho)\DA \& \ n_s < \text{use}\, M^{D_s}_{q_s}(\rho) \le q_{s-1}\right\rrbracket\\
&=& \sum_\rho 2^{-|\rho|} \left\llbracket M^{A}(\rho)[q_s]\DA \& \ m_s < \text{use}\, M^{A}(\rho)[q_s] \le q_{s-1}\right\rrbracket\\
\end{eqnarray*}
So by our choice of $q_0 < q_1 < \dots$, $\sum_s \cost^{D_s}(n_s, s) < \infty$.
\end{proof}

  \section{Nies:  Shannon-McMillan-Breiman theorem and its non-classical versions}  
 
 This section is based on discussions with Marco Tomamichel and others, and  on my   talk at the M\"unster department of mathematics colloquium in Jan 2018 where I thank the audience for an unusually lively response during the talk.
 
 \subsection{Shannon's work in information theory}

 \newcommand{\tr}{\mbox{\rm \textsf{Tr}}}
 \newcommand{\Malg}[1]{M_{\tp{#1}}}

 We are given a source emitting symbols from an alphabet $\mathbb A = \{ a_1, \ldots, a_n\}$. The symbol $a_i$ has probability $p_i$. In Shannon's  original  work the symbols are emitted independently.    So this can be modelled by a sequence of i.i.d.\ $\mathbb A$-valued random variables. 
 
 We want to encode  a string of $n$ symbols by a bitstring,  using as few bits as possible. 
 However, it is allowed that   certain  strings are not encoded at all, as long as the probability of this happening goes to $0$ with $n \to \infty$. Let  $k_n$ be  the number of bits we allow for encoding  $n$ symbols. The asymptotic compression rate   is $h= \liminf k_n/n$. What is the least $h$ we can achieve?

 Shannon's  \emph{source coding theorem} says that $h$ is the entropy of the probability distribution: $-\sum_{i} p_i \log p_i$. As we encode symbols strings by bit strings, the log is taken in base $2$. 
 
 To prove that $h$ is an upper bound, for given $\epsilon$ one considers the set $A_{n, \epsilon} $ of $\epsilon$-typical strings, namely those $u \in \mathbb A^n$ such that $\log \mathbb P[u]$ (where $\mathbb P[u]$ denotes the probability that $u$ happens) is within $\epsilon$ of $hn$. One shows that   the  probability of $A_{n, \epsilon} $ goes to $1$ as $n \to \infty$, and that the size of $A_{n, \epsilon} $ is at most $\tp{n(h+ \epsilon)}$. So we need at most $n(h+ \epsilon)$ bits to encode such a string. For $\epsilon \to 0$ we need $hn$ bits. 
 
  \subsection{The Shannon-McMillan-Breiman theorem}
  The SMB theorem generalises the above to the case that the r.v.'s $X_n$ ($n \in \NN $) form an ergodic process, to be defined below.   It says that the entropy of the joint distribution can be seen from almost every trajectory $\omega$ as the limit of the empirical entropy $h_n(\omega)$, a random variable (r.v.)  defined below. 
  
  The theorem   is based on three separate papers: Claude Shannon  1948, Brockway McMillan 1953, Leo Breiman 1957.  The former two worked in the spirit of information theory. Shannon only did the case of a Markov process (which includes the i.i.d. case).  McMillan's paper is long and follows the notation and terminology introduced by Shannon.  He obtained the stronger $L_1 $ convergence.  Breiman's paper is a short addition using an important inequality of McMillan's; he shows a.e.\ convergence of the r.v.'s $h_n$ defined below. (There is also an erratum because some calculation on the last page wasn't right.)  
  
  We follow Shields' book \cite{Shields:96} in the exposition of the theorem, though we adapt some notation.
    $\mathbb A^\infty$ denotes the space one-sided infinite sequences of symbols in $A$. We can assume that this is the sample space, so that $X_n(\omega) = \omega(n)$.  By $\mu $ we denote their joint distribution.  A dynamics on $\mathbb A^\infty$  is given by the shift operator $T$, which erases the first symbol of   a sequence.
  A measure $\mu$ on $A^\infty$ is \emph{$T$-invariant} if $\mu G = \mu T^{-1}(G)$ for each measurable $G$. 
  
  By $\omega \uhr n$ we denote the first $n$ symbols of $\omega$. Note that in ergodic theory one usual starts with $1$ as an index, so they would write $\omega_1^n$ for the first $n$ symbols. Here the notation should be compatible with the one used in randomness theory.) 
  
  We consider the r.v.  \[ h_n(\omega ) = -\frac 1 n \log \mu [\omega\uhr n],\]  (recall that  $\log$ is w.r.t.\ base $2$). 
The main thing to prove is the following fact for general $T$-invariant measures. 
\begin{lemma} \label{lem: limit exists} Let $\mu $ be an     invariant measure for the shift operator $T$ on the space $A^\infty$.  Then for  $\mu$-a.e. $x$,
$h(x) = \lim_n h_n(x\uhr n)$ exists. \end{lemma}

Recall that $\mu$ is \emph{ergodic} if   every  $\mu$ integrable function $f $ with $f \circ T = f$  is constant $\mu$-a.s.  
An equivalent  condition that is easier to check is the following: for  $u,v \in \mathbb A^*$,  
\[ \lim_N \frac 1 N \sum_{k=0}^{n-1} \mu ([u] \cap T^{-k}[v]) = \mu [u] \mu [v]. \]
(Don't confuse this with the stronger property called ``weakly mixing", where one requires   that the average of the absolute  values of the differences goes to 0; this happened during  the talk.) It is easily seen that the Bernoulli measure on $\mathbb A ^\infty$ is ergodic (and in fact, strongly mixing).

For ergodic $\mu$, the entropy $H(\mu)$ is defined as $\lim_n H_n(\mu)$, where \[H_n(\mu) =  -\frac 1 n \sum_{|w| = n}  \mu [w] \log \mu [w].\]
One notes that $H_{n+1}(\mu) \le H_n(\mu)\le 1$ so that the limit exists. Also note that $H_n(\mu) = \mathbb E h_n$.

The following says that in the ergodic case, $\mu$-a.s.\ the empirical  entropy equals the entropy of the measure.

\begin{thm}[SMB   theorem] Let  $\mu $ be an  ergodic    invariant measure for the shift operator $T$ on the space $A^\infty$. 
Then for $\mu$-a.e.\  $\omega$   we have  $\lim_n h_n(\omega) = H(\mu)$.
\end{thm}
Given Lemma~\ref{lem: limit exists},  this isn't too much extra work to prove. First one checks that since $\mu$ is $T$-invariant, we have $h(Tx) \le h(x)$ for each $x$, i.e.\ $h$ is subinvariant. Next, from  the Poincare recurrence theorem it follows that  $B= \{ x \colon h(Tx) < q <  h(x)\} $ is a null set for each $q$ (because we can't return to $B$ outside a null set), so $h$ is actually invariant: $h(Tx) = h(x)$ for $\mu$-a.e.\ $x$. Also $h \in L^1(\mu)$ by the dominated convergence theorem, so if $\mu$ is ergodic  then $h(x) $ has  some constant value, for $\mu$-a.s. $x$. 
A final step is then to show that this constant value equals $H(\mu)$. 

%

\subsection{Proof in the i.i.d. case} It is instructive to give a direct proof of the SMB theorem in the i.i.d. case, that is, when $\mu$ is a Bernoulli measure. 
 Suppose symbol $a_i$ has probability $p_i$. By $k_{i,n}(\omega)$ we denote the number of occurences of the symbol $a_i$ in $\omega\uhr n$. By independence,  we have 
 
 \[ h_n(\omega) = - \frac 1 n \log \prod_i   p_i^{k_{i,n}(\omega)} = - \frac 1 n \sum_i k_{i,n}(\omega) \log p_{i}.\]
 By the strong law of large numbers, for $\mu$-a.e.\ $\omega$, $k_{i,n}(\omega) /n$ converges to $p_i$.

\subsection{Algorithmic version of the SMB theorem}
 We now assume that we can compute $\mu [u]$ uniformly from a string $u$. That is, $\mu$ is a computable measure. This is true e.g. for the Bernoulli measure when the $p_i$ are all computable reals. 
 
 Hochman \cite{Hochman:09} and in more explicit form Hoyrup \cite{Hoyrup:12} have shown that the exception set in the SMB theorem is ML-null.  It is unknown at present whether a weaker randomness notion such as Schnorr's is sufficient, even  under the assumption that  $H(\mu)$ is computable.

%
%
%
%

\subsection{Do random states satisfy the  quantum  SMB theorem?} 

Mathematically, a qubit is a  unit vector in the Hilbert space $\mathbb C^2$. We  give  a brief summary  on   ``infinite  sequences" of qubits. One considers the $C^*$   algebra $M_\infty =  \lim_n M_{2^n}(\mathbb C)$, an  approximately finite  (AF) $C^*$ algebra. ``Quantum Cantor space" consists of the state set $\+ S(M_\infty)$, which is a convex, compact, connected set with a shift operator, deleting the first  qubit. 

Given a finite  sequence of qubits, ``deleting"  a particular one generally  results in a statistical superposition of the remaining ones. This is is why $\+ S(M_\infty)$ consists of coherent sequences of density matrices in $M_{2^n}(\mathbb C)$ (which formalise such superpositions) rather than just of sequences of unit vectors in $(\mathbb C^2)^{\otimes n}$. For more background on this, as well as  an algorithmic  notion of randomness   for such states,  see Nies and Scholz \cite{Nies.Scholz:18}. Notice that the tracial state $\tau$ is random, even though it generalises the uniform measure and hence, from a different point of view,  can be  considered to be  computable. 

Bjelakovich et al. \cite{Bjelakovic.etal:04} provided a quantum version of the Shannon-McMillan theorem. (They worked with bi-infinite sequences, which makes little difference here, as a stationary process is given by its marginal distributions on the places from $0$ to $n$, for all $n$.)
The reason  they avoided the full  Breiman version  is that on $\+ S(M_\infty)$ there has been so far no reasonable way to say  ``for almost every". (The work in \cite{Nies.Scholz:18} introduces effective null sets, which might   remedy this.)  In  \cite{Bjelakovic.etal:04}, they first convert the classical SMB theorem  into an equivalent form  which doesn't directly mention measure; rather, they have ``chained typical sets" which generalise Shannon's typical sets.  To be chained means that  they are coherent over successive  lengths of symbol strings. 

The von Neumann entropy of a density matrix $S$ is $H(S) = - \tr (S \log S)$. For a state $\mu$ on $M_\infty$ we let 

\[ h(\mu) =  \lim \frac  1 n H(\mu \uhr {M_{\tp{n}}})\]
which exists by concavity of $\log$. 
Let $\mu $ be a  state on $M_\infty$. For a quantum $\SI 1$ set $G= \seq {p_n}\sN n$ we define $\mu(G) = \sup_n \tr ( \mu\uhr  {M_{\tp{n}} } p_n)$. A qML-test relative to a computable state  $\mu$ is a uniform sequence $(G_r)$ of such sets such that $\mu(G_r) \le \tp{-r}$. Failure and passing is defined as before. This yields qML-randumness w.r.t.\ $\mu$.  Work in progress with Tomamichel would show the following.

\begin{conjecture} Let $\mu$ be an ergodic computable state on $M_\infty$. Let $\rho$ be a state that is quantum ML-random   with respect to $\mu$. Then

\[   h(\mu) = -  \lim \frac 1 n\tr (\rho \uhr  {M_{\tp{n}}}  \log  \mu \uhr  {M_{\tp{n}}}). \] \end{conjecture}

The plan is to go through  more and more general cases for both $\rho$ and $\mu$.  The computable state $\mu$ can be uniform (ie $\tau$), i.i.d. but quantum, a computable ergodic measure, and finally any computable ergodic state. The random state $\rho$ can be a bit sequence that is ML random wrt $\mu$, a $\mu$-random measure on $\tp{\NN}$, and finally any qML($\mu$) state.  The combination that  $\rho$ is a bit sequence, and $\mu $ a measure is the effective classical SMB theorem, essntially proved by  Hochman \cite{Hochman:09} and in more explicit form by Hoyrup \cite{Hoyrup:12}. 

In the classical setting the  case where $\mu$ is  a Bernoulli measure is  easy. In the quantum setting we use Chernoff bounds and some calculations to do the case for general $\rho$   but Bernoulli $\mu$. 

  To say that $\mu$ is i.i.d. means that for some fixed computable $V \in S(M_2)$,  i.e. a 2x2 density matrix,  we have $\mu\uhr  {M_{\tp{n}}} = V ^{\otimes n}$.  Note that the partial trace removes the final $V$, so this ``infinite tensor power" indeed can be seen as  a computable state on $M_\infty$. There is a computable unitary $U \in M_2$ such that $U V U^\dagger$ is diagonal, with $p$, $1-p$ on the diagonal, $p$ is computable. Its von Neumann entropy is $h(\mu) = - p \log p - (1-p) \log (1-p)$.  

Note that  qML($\mu$)-randomness is closed under the unitary of $\Malg \infty$ which is obtained applying  conjugation by $U^\dagger$ ``qubit-wise". So replacing $\rho$ by its conjugate we may as well assume that $V$ is diagonal. Fix $\delta >0$. Let $P_{n, \delta}$ be the projector in $\Malg  n$ corresponding to the set of bitstrings

\bc $\{ x \colon |x|  = n \land | - \frac 1 n \log \mu[x] - h(\mu)| \le \delta \}
$.  \ec   

Since $\mu$ is a product measure, $\log \mu[x] $ is a sum of $n$  independent random variables looking at the bits of $x$, and the expectation of $- \frac 1 n \log \mu[x]$ is $h(\mu)$. The usual Chernoff bound yields $\mu (P_{n, \delta}^\perp) \le 2 \exp(-2n \delta^2)$. Let $G_{m, \delta} = \bigcup_{n>m}  P^\perp_{n,\delta}$ where these projectors are now viewed as clopen sets in Cantor space, so that $G_{m, \delta}$ determines a classical ML-test. Since $\rho$ is qML random w.r.t.\ $\mu$, we have $\lim_m \rho(G_{m, \delta}) =0$. 

\begin{thm} $\lim_n  -\frac 1 n \tr(\rho \uhr {\Malg n} \log \mu \uhr {\Malg n}) = h(\mu)$. \end{thm}
To see this, fix $\delta > 0$,  and omit it from the subscripts for now. We write  $s= h(\mu)$ and write $\rho_n$ for $\rho \uhr {\Malg n}$ etc.

We insert the term $I_{2^n}= P_n^\perp +  P_n$ between the two factors. We  look  separately at both resulting limits. 

\n \emph{Part  1.} We consider $-\frac 1 n \tr(\rho_n P_n^\perp \log \mu_n)$. Note that $P_n^\perp \log \mu_n$ is negative semidefinite as the two factors  are diagonal w.r.t.\ the same base and therefore commute. So $ - \frac 1 n \tr (\rho_n P_n^\perp \log \mu_n) \ge 0$. By cyclicity of the trace and the commutation, we have 
\bc $\tr (\rho\uhr n P_n^\perp \log \mu_n )= \tr (\rho\uhr n P_n^\perp \log \mu_n P_n^\perp ) =  \tr (P_n^\perp \rho_n P^\perp_n \log \mu_n )$. \ec For positive operators $A,B$ we have $\tr (AB) \le  ||A||_1 \cdot ||B||_\infty$ where $||A||_1 $ is the sum of the      eigenvalues, and $||B||_\infty$ is their maximum.  So 

\bc $ - \frac 1 n \tr (\rho_n P_n^\perp \log \mu_n) \le \frac 1 n ||P_n^\perp \rho_n P_n^\perp ||_1 \log \mu_n ||_\infty$. \ec
Now $ ||\frac 1 n \log \mu_n ||_\infty$ is bounded depending only on $p$, and for large enough $n$  we have  $||P_n^\perp \rho_n P_n^\perp||_1 \le 2 \delta$ by hypothesis. 

\vsp\n \emph{Part 2.}  We consider $-\frac 1 n \tr(\rho_n P_n \log \mu_n)$. We note that $\mu_n$ is a diagonal matrix in $\Malg n$ where the entry in the position $(\sss, \sss)$ is $p^k (1-p)^{n-k}$, where the binary string $\sss$ of length $n$ has $k$ 0s. By definition of $P_n$ it follows that $||P_n (- \frac 1 n \log \mu_n) - h(\mu) P_n ||_\infty \le \delta$. Now 

\begin{eqnarray*} - \frac 1 n \tr (\rho_n P_n \log \mu_n) &=& \tr \rho_n (- \frac 1 n P_n \log \mu_n - sP_n + sP_n)  \\
& = & s \tr \rho_n P_n + \tr (\rho_n (- \frac 1 n P_n  \log \mu_n - sP_n)). \end{eqnarray*}

Using that $\tr (AB) \le  ||A||_1 \cdot ||B||_\infty$ for positive $A, B$ and the definition of $P_n$, the  second   summand is   at most $\delta$.

By hypothesis,  for large $n$ we have $\tr ( \rho_n P^\perp_n) \le  2 \delta$, and hence $\tr \rho_n P_n \ge 1- 2\delta$. So  the first summand is between  $s(1-2\delta)$ and   $s$. 

To summarize, for large $n$, the quantity in Part 1 is $\le 2\delta$ and the quantity in Part 2 is in $[s(1-2\delta), s+ \delta]$. For $\delta \to 0$ their sum converges to $s$ as required.

\subsection{Random states satisfy the  law of large numbers}
\begin{prop} Let $\mu$ be an i.i.d computable state,   and let $\rho$ be qML-random relative to $\mu$. For $i<n$ let $S_{n,i}$ be the subspace of $\mathbb C^{2^n}$ generated by those $\underline \sigma$ with $\sigma_i =1$. We have \bc $\lim_n \frac 1 n \sum_{i< n} \tr (\rho_n S_{n,i}) = p$ \ec where $S_{n,i}$ is identified with its orthogonal projection.  \end{prop}

\begin{proof} As above,  we first assume that \[ \mu = \left (  \begin{matrix} p & 0 \\ 0  & 1-p \end{matrix} \right )^{\otimes \infty}  \]  corresponds to a classical product measure where $p\in (0,1)$ is computable.  As before we fix $\delta >0$ .  Let $E_{n, \delta}$ be (the projector in $\Malg  n$ corresponding to) the set of bitstrings

\bc $\{ x \colon |x|  = n \land | - \frac 1 n  \sum x  - p| \le \delta \}
$.  \ec    The Chernoff bound now yields $\mu (E_{n,\delta}^\perp) \le 2 \exp(-2n \delta^2)$

 Let $G_{m, \delta} = \bigcup_{n>m}  E^\perp_{n,\delta}$ where these projectors are now viewed as clopen sets in Cantor space, so that $G_{m, \delta}$ determines a classical ML-test. Since $\rho$ is qML random w.r.t.\ $\mu$, we have $\lim_m \rho(G_{m, \delta}) =0$. 
 Now,  
 $\frac 1 n \sum_{i< n} \rho(S_{n,i}) =$ 
 \bc $ \frac 1 n\sum_{x \in E_{n, \delta} }\sum_{i< n} \rho( S_{n,i }  \cap [x] )  + \frac 1 n\sum_{x \in E^\perp_{n, \delta}} \sum_{i< n} \rho( S_{n,i }  \cap [x] )$.  \ec

 For large enough $n$, the second summand is $\le 2 \delta$. The first summand equals $\frac 1 n \sum_{x \in E_{n, \delta}} (\sum x) \rho[x]$, where $\sum x$ is the number of 1s in a string $x$. By definition of $E_{n, \delta}$ this value is in $(p-\delta, p+ \delta) \rho(E_{n, \delta})$, and $ \rho(E_{n, \delta}	$ tends to $1$ with $n\to \infty$. Letting $\delta \to 0$ we get the value $p$.
 
 For general i.i.d.\ states  $\mu$, we note that the same argument works for other computable orthonormal  bases  $e_0, e_1$ of $\mathbb C^2$ instead of $| 0\ra, |1\ra$. If $\mu$ is the infinite tensor power of $ UB(p) U^\dagger$, we conjugate  $\rho$ qubitwise by   $U^\dagger$ and carry out the argument for $e_r = U |r\ra U^\dagger$.

\end{proof}

 
  \part{Reverse mathematics}
 \section{Carlucci: A variant of Hindman's Theorem implying $\mathsf{ACA}_0'$}

The following is part of an attempt to prove (or disprove) the existence of a level-by-level 
combinatorial reduction from Ramsey's Theorem to Hindman's Theorem. Such a reduction 
would be a strong way of establishing that Hindman's Theorem implies $\mathsf{ACA}_0'$.

Let us recall some definitions. 

\begin{definition}[Hindman's Theorem with bounded sums]
For positive integers $n,\ell$, $\mathsf{HT}^{\leq n}_\ell$ denotes the following principle: 
for every coloring $f:\mathbb{N}\rightarrow \ell$ there exists an infinite set $H$ 
such that $FS^{\leq n}(H)$ is monochromatic for $f$, where $FS^{\leq n}(H)$ denotes the 
set of all non-empty finite sums of at most $n$ distinct members of $H$. 
\end{definition}

The analogous version for sums of exactly $n$ many terms is denoted $\mathsf{HT}^{=n}_\ell$.

\begin{definition}[Apart set]
A subset $X$ of the positive integers is apart if for any $n,m\in X$ such that $n<m$, 
the largest exponent of $n$ in base 2 is strictly smaller than the smallest exponent of 
$m$ in base 2.
\end{definition}

If $P$ is an Hindman-type principle then $P$ {\em with apartness} denotes the same principle to which 
we add the requirement that the solution set is an apart set. 

Recently the following results where established, where $\mathsf{RT}^k_\ell$ denotes Ramsey's Theorem for 
exponent $k$ and $\ell$ colors, $\mathsf{IPT}^2_2$ denotes Dzhafarov and Hirst's \cite{Dzhafarov.Hirst:11}
Increasing Polarized Ramsey's Theorem for exponent $2$ and $2$ colors, and $\leq_{\mathrm{sc}}$ 
denotes strong computable reducibility.

\begin{enumerate}
\item For every positive integers $k,\ell$, $\mathsf{RCA}_0 \vdash \mathsf{HT}^{\leq 3}_3 \to \mathsf{RT}^k_\ell$ (Dzhafarov et al. \cite{Dzhafarov.Jockusch.etal:16}),
\item For every positive integers $k,\ell$, $\mathsf{RCA}_0 \vdash \mathsf{HT}^{\leq 2}_4 \to \mathsf{RT}^k_\ell$ (Carlucci et al. \cite{Carlucci.Kolodziejczyk.etal:arxiv}),
\item For every positive integers $k,\ell$, $\mathsf{RCA}_0 \vdash \mathsf{HT}^{= 3}_2 \mbox{ with apartness } \to \mathsf{RT}^k_\ell$ (Carlucci et al. \cite{Carlucci.Kolodziejczyk.etal:arxiv}),
\item $\mathsf{IPT}^2_2 \leq_{\mathrm{sc}} \mathsf{HT}^{\leq 2}_4$ (Carlucci \cite{Carlucci:arxiv}),
\item $\mathsf{IPT}^2_2 \leq_{\mathrm{sc}} \mathsf{HT}^{= 2}_2 \mbox{ with apartness}$ (Carlucci et al. \cite{Carlucci.Kolodziejczyk.etal:arxiv}).
\end{enumerate}

Despite points (1) and (2), I have not been able to lift the combinatorial reductions in points (3) and (4) 
to exponents higher than $2$. Below I show that it is possible to 
do so, so as to hit Ramsey's Theorem and not only its increasing polarized version, 
provided one adds an extra condition on the elements of the solution set to Hindman's Theorem. 
The resulting variant of Hindman's Theorem is then shown to be equivalent to Ramsey's Theorem.

\begin{definition}[Exactly large number]
A positive integer $n$ is $!\alpha$-large (exactly $\alpha$-large) if the set $e(n)$ of its exponents
in base 2 is $!\alpha$-large.
\end{definition}

Recall that a set $X$ of positive integers is $!\omega$-large if $|X|=\min(X)+1$, 
is exactly $\omega 2$-large is it has the form $X=X_1\cup X_2$ with $\max(X_1)<\min(X_2)$ 
such hat $X_1$ and $X_2$ are exactly $\omega$-large, and so on for $!\omega 3$, $!\omega 4$ etc.

\begin{definition}
Let $A^k_\ell$ be the following principle: For every coloring $c$ of the positive integers $\NN$ in $\ell$ colors
there exists an infinite subset $H$ of $\NN$ such that each $n\in H$ is $!\omega$-large, $H$ is apart, 
and $FS^{=k}(H)$ is monochromatic.
\end{definition}

The principle $A^k_\ell$ is essentially $\mathsf{HT}^{=k}_\ell$ with apartness plus the extra (far from trivial!) constraint that the 
solution set is contained in the set of $!\omega$-large numbers.

\begin{theorem}
$\mathsf{RCA}_0 \vdash \forall k > 0\forall \ell>0 (A^k_\ell \to \mathsf{RT}^k_\ell)$. In fact, $\mathsf{RT}^k_\ell\leq_{\mathrm{sc}} A^k_\ell$. 
\end{theorem}

\begin{proof}
Let $d:[\NN]^k\to \ell$ be given. Define $c:\NN\to \ell$ as follows. 
If $n$ is not $!\omega k$-large then $c$ colors $n$ arbitrarily. 
If $n$ is $!\omega k$-large then $c(n)=d(n_1,\dots,n_k)$, where $n_1 < \dots < n_k$ 
are the unique $!\omega$-large numbers such that $n=n_1+\dots+n_k$.
Let $H$ be a solution to $A^k_\ell$ for $c$ and let $i<\ell$ be the color of $FS^{=k}(H)$
under $c$. We claim that $[H]^k$ is monochromatic of color $i$ under $d$.
Indeed, let $a_1<\dots<a_k$ be in $H$. Then $a_1,\dots,a_k$ are $!\omega$-large and add in base 2 with 
no carry since $H$ is apart. Thus, $a=a_1+\dots+a_k$ is $!\omega k$-large and is in $FS^{=k}(H)$.
Therefore $i = c(a)=d(a_1,\dots,a_k)$. 
\end{proof}

\begin{theorem}
$\mathsf{RCA}_0\vdash \forall k > 0\forall \ell>0 (\mathsf{RT}^k_\ell \to A^k_\ell)$. 
In fact, $A^k_\ell\leq_{\mathrm{sc}} \mathsf{RT}^k_\ell$.
\end{theorem}

\begin{proof}
Let $c:\NN\to \ell$ be given. Define $d:[\NN]^k\to \ell$ as $d(a_1,\dots,a_k)=c(a_1+\dots+a_k)$.
Let $X\subseteq \NN$ be an infinite apart set consisting of $!\omega$-large numbers.
By Ramsey's Theorem for $X$ and $d$, there exists an infinite $H\subseteq X$ such that
$d$ is constant on $[H]^k$, say of color $i<\ell$. Then $H$ is a solution to $A_k$ for $c$,
since, if $a\in FS^{=k}(H)$ then $a$ is a sum of $k$ many exactly large elements of 
$H$, i.e., for some $a_1 < \dots < a_k$ in $H$ we have that $a=a_1+\dots+a_k$. 
Then $i = d(a_1,\dots,a_k)=c(a)$.
\end{proof}

\begin{corollary}
$\mathsf{RCA}_0 \vdash \forall k A^k_2 \to \forall k \mathsf{RT}^k_2$.
\end{corollary}

Hence, $\forall k A^k$ implies $\mathsf{ACA}_0'$ over $\mathsf{RCA}_0$. The following question is then of interest:

\begin{question}
Does Hindman's Theorem imply $\forall k A^k_2$? 
\end{question}
Note that, in $A^k_\ell$, the condition that the elements of the solution set are $!\omega$-large can be replaced
by various other conditions. For example we might require that all elements of the solution set
have the same binary length. The argument showing that $A^k_\ell$ implies $\mathsf{RT}^k_\ell$ is inspired by 
an argument attributed to Justin Moore which I learned from David Fernandez Breton
(private communication), proving that a cardinal satisfying Hindman's Finite Unions Theorem has to be
weakly compact. 


  \part{Group  theory and its connections to logic}
     \section{Chiodo, Nies and Sorbi: Decidability problems for f.g.\ groups}

Maurice Chiodo, Nies and Andrea Sorbi discussed decidability problems for f.g.\ groups in April and May, both in Siena and by Skype. 

\subsection*{The c.e.\ equivalence relation of isomorphism between finitely presented (f.p.)  groups} C.F.\ Miller \cite{Miller:71}  has shown that the  c.e.\ equivalence relation (ceer)  $\cong_{f.p.}$ of isomorphism between finitely presented groups is $\SI 1$ complete within the ceer's. Nies and Sorbi \cite{Nies.Sorbi:16} noticed that it has a diagonal function $f$, namely $f$ is computable and $f(x)$ is not equivalent to $x$. 

The ceer $\cong_{f.p.}$   is not effectively inseparable as pointed out by Chiodo: Let $\+ A$ be the class of f.p.\ groups $G $ such that   $G_{ab} \cong  \ZZ$. Then $\+ A $ is computable, and separates the class of  a presentation of $\ZZ$ from the class of a presentation of $\ZZ \times \ZZ$.  

Recall that a group $G$ is perfect if $G' = G$, or equivalently $G_{ab}$ is trivial. We can list   finite presentations of all the  perfect f.p.\ groups by including relations that write each generator as a product of commutators in a particular way. 

\begin{question} If  $\cong_{f.p.}$ restricted to presentations of perfect groups effectively inseparable? \end{question}

\subsection*{Problems on f.g.\ groups}

The following questions are  long standing.
\begin{question} Is some  infinite f.p.\ group  a torsion group? \end{question}

\begin{question}[Weigold] Is each f.g.\ perfect group the normal closure of a single element? \end{question} 

\subsection*{Algorithmic problems}

\begin{question} Is the relation  among f.p.\ groups ``$B$ is a quotient of $A$" $\SI 1$--complete as a pre-order? 
\end{question} 

\begin{question} Find an algorithm that on input a finite presentation $G=\la X \mid R \ra$, outputs a word $w$ in $X$ such that $w= 1 $ in $G \LR G $ is trivial. \end{question} 
     \section{Fouche and Nies: randomness notions in computable profinite groups}

Willem Fouche visited New Zealand for three weeks in October. He and Nies continued their work on the effective content of  results of Jarden, Lubotzky and others. This work was started in  \cite[Section 16]{LogicBlog:16}, where background is provided. We only recall here the following.

\begin{definition}[Smith~\cite{Smith:81}] \label{def: computable profinite} \  \bi \item[(i)]   A profinite group $G$ is called \emph{co-r.e.} if it is the inverse limit  of a computable  inverse system $\la G_n, p_n\ra$ of finite groups  (i.e.\ the groups $G_n$ and the maps $p_n$ between them are uniformly computable). Equivalently, the subgroup $U$  above is a $\PI 1$ subclass of $\prod_{n} G_n$. \item[(ii)] $G$  is called \emph{computable} if, in addition,  the  maps $p_n$    can be chosen  onto. In other words, the set of extendible nodes in the tree corresponding to $U$ is computable.   \ei\end{definition}

Each separable profinite group is equipped with a unique Haar probability measure (i.e., a  probability measure that is invariant under left and under right translations). For instance, for the 2-adic integers $\ZZ_2$, the Haar measure is the usual product measure on Cantor space.
If the group is computable then so is the Haar measure, using the notion of a  computable probability  space due to  Hoyrup and Rojas \cite{Hoyrup.Rojas:09a}.

The  Jarden, Lubotzky  et al.\ results are  theorems of ``almost everywhere" type in various   profinite groups $G$: they assert a property for almost every tuple in $G^e$, for some  $e$ that is fixed for the particular result. These groups are usually computable, in which case randomness notions defined via algorithmic tests (with respect  to the Haar measure) can be applied in $G^e$.  So,  unlike the usual process of effectivizing results from     analysis \cite{Brattka.Miller.ea:16}, in this case the existing ``classical" results have an effective content \emph{per se}, which  only needs to be made explicit. To do so is our purpose. 

\subsection{Computable profinite groups that are completions} \label{ss:computable completions}
 We update the information in  \cite[Section 16]{LogicBlog:16}. The definition  below is taken from  \cite[Section 3.2]{Ribes.Zalesski:00}. 
 Let $G$ be a group, $\+ V$ a set of normal subgroups of finite index in $G$ such that $U, V \in  \+ V $ implies that there is $W \in \+ V$  with  $W \sub U  \cap V$. We can turn $G$ into a topological group by declaring $\+ V$ a basis of neighbourhoods (nbhds) of the identity. In other words, $M \sub G$ is open  if for each $x \in M$ there is $U \in \+ V$ such that $xU \sub M$. 
 
\begin{definition} \label{def:completion} The completion of $G$ with respect to $\+ V$ is the inverse limit  $$G_{\+ V} = \varprojlim_{U \in \+ V} G/U,$$ where $\+ V$ is ordered under inclusion and the inverse system is equipped with  the natural    maps: for $U \sub  V$, the  map  $p_{U,V} \colon G/U \to G/V $  is given by $gU \mapsto gV$. 
\end{definition} The inverse limit can be seen as a closed subgroup of the direct product $\prod_{U \in \+ V} G/U$ (where each group $G/U$ carries  the discrete topology), consisting of the functions~$\alpha $ such that $p_{U,V}(\alpha (gU)) = gV$ for each $g$.  Note that the map $g\mapsto (gU)_{U \in \+ V}$ is a continuous homomorphism $G \to G_{\+ V}$ with dense image; it is injective iff $\bigcap \+ V = \{1\}$. 
  
  Suppose $G$ is a computable group, and the class $\+ V$ in Definition~\ref{def:completion} is uniformly computable in that there is a uniformly computable sequence  $\seq{R_n}$ such that $\+ V = \{ R_n \colon n \in \NN \}$. Suppose further that  $W$ above can be obtained effectively from $U,V$. Then there is a  uniformly computable  descending  subsystem  $\seq {T_k}$ of $\seq {R_n}$  such that  $\fa n \ex k \, T_k \le R_n$. Since we can effectively find coset representatives of $T_n$ in $G$, the inverse system  $\seq  {G/T_n}$ with the natural projections  $T_{n+1} a \to T_n a$ is computable. So $G_\+ V$ is computable.  
 
Suppose we are given two computable sequences $\seq{R_n}$ and $\seq{S_k}$ as above. If for each $n$ we can compute $k$ such that $S_k \sub R_n$, and vice versa. Then the completions obtained via the two sequences are computably isomorphic. 
 
 The criterion above is satisfied by $F_k$ and $F_\omega$ with the systems of normal subgroups introduced in \cite[Section 16]{LogicBlog:16}. Thus their completions $\hat F_k$ and $\hat F_\omega$ are computable profinite groups. 
 
 \begin{lemma} Let $G$ be $k$-generated ($ k \le \omega$). Then $G$ is computable [co-r.e.] iff $G= \hat F_k/N$ for some computable normal subgroup $N$ ($\PI 1$ N).  \end{lemma}

\subsection{Abelian free profinite groups}

Let $\widehat \ZZ$ denote the free profinite group of rank $1$.  Note that $\widehat \ZZ$ is the inverse limit of the directed system of groups $\ZZ / n \ZZ$ with the natural projections from $\ZZ / n \ZZ$ to $\ZZ / k \ZZ$ in case $k$ divides $n$.  By $\la S \ra$ one denotes the closed subgroup generated by a subset $S$ of a group.

\begin{prop} Let $z \in \widehat   \ZZ $ be Kurtz random. Then    $\la z \ra$ has infinite index in $\widehat \ZZ$ and $\la z \ra \cong \widehat \ZZ$.  
\end{prop} 

For tuples rather than singletons, the opposite happens: random tuples generate a subgroup of finite  index.

\begin{prop} Let $e \geq 2$ and suppose that $z \in (\widehat   \ZZ)^e $ is  Schnorr random.  Then $\la z \ra$ has  finite index in $\widehat \ZZ$ and $\la z \ra \cong \widehat \ZZ$. 
 \end{prop}
 \begin{proof} Let $G_n = (n \widehat \ZZ)^e$. Note that $\mu G_n = n^{-e}$ and $G_n$ is uniformly $\SI 1$. Since $\sum_n n^{-e}$ is finite and computable,  $\seq {G_n}\sN n$ is a Schnorr-Solovay test. Therefore $ z \not \in G_n$ for sufficiently large $n$.  
 
 If $U$ is a closed subgroup of $ \widehat \ZZ$ then $U$ is the intersection of the groups of the form  $n \widehat \ZZ$ containing it. So,  if $U$ has infinite index there are infinitely many such $n$. Hence $\la z \ra$ has finite index.  Since it is also  closed, it is open, and hence isomorphic to $\widehat \ZZ$. 
\end{proof}

\begin{prop} There is a Kurtz random $z \in (\widehat   \ZZ)^2 $ such that $\la z \ra$ is of infinite index in $\widehat \ZZ$. \end{prop}
\begin{proof} Let $G_n = \ZZ / n! \ZZ$. Clearly $\hat \ZZ = \projlim G_n$. Furthermore, by Subsection~\ref{ss:computable completions}, the corresponding computable presentation of $\widehat \ZZ$ given by this inverse limit is computably isomorphic to the standard one obtained as the completion of $\ZZ$. 

We can now view an element of $w \in \hat \ZZ$ as written in factorial expansion 
\[ w = \sum_{n \ge 1}  a_n n! \]
where $0 \le a_n \le n$. In this way we can think of $w$ as a path $f = (a_1, a_2, \ldots)$ on the tree $T$ where every node at level $k$ (starting at level $0$ for the root) has $k+1$ children.  The path space  $[T]$ is homeomorphic to $\widehat \ZZ$ via a map  turning the uniform measure on $[T]$ into the Haar measure on $\widehat \ZZ$.

Given $z = (z^0, z^1)$, let $f^0, f^1$ be the corresponding paths. If there are infinitely many $n$ such that $f^0(n)= f^1(n) =0$, then the subgroup of $\widehat \ZZ$  topologically generated by $z$  has infinite index. 

A pair $(f,g) \in [T]^2$  is called \emph{weakly 1-generic} if it meets each dense $\SI 1$ set. In particular it meets the condition above.
Any weakly 1-generic is Kurtz random. Thus the pair $z$ corresponding to a weakly 1-generic pair of paths is Kurtz random and generates a subgroup of infinite index, as required.
\end{proof}

\subsection{Normal closure in general computable profinite groups}

Let $G$ be a topological  group. For $z \in G^e$, by $[z]_G$ one denotes the topological normal closure of the tuple $z$ in $G$.  We omit the subscript if it is clear from the context.  Jarden and Lubotzky also  consider almost everywhere theorems  of the form $[z]_G \cong L$ for a.e.\ $z$, where $L$ is an appropriate  profinite group; specifically, $L$ can be recognized by its finite quotients  among the closed normal subgroups of $G$ (e.g., $L$ could be free of a certain rank). We will show that weak 2-randomness of $z$ suffices.

\begin{lemma} Let $G$ be a computable profinite group. Fix $e \in \NN$ and a finite group $C$.  
The set $ \{ z \in G^e \colon \, [z] \text{ has $C$ as a quotient } \} $
is a $\Sigma^0_2$ subset of $G^e$. \end{lemma}
  
  \begin{proof}  According to Definition~\ref{def: computable profinite} $G = \projlim_n G_n$ for a computable  inverse system $\la G_n, p_n\ra$   of finite groups with onto maps $p_n$.  By $g\uhr n$ we denote the projection of $g\in G$ into $G_n$; similar notation applies to $g \in G_t$ for $t \ge n$. For $g\in G$, we have 
  
  \bc $g \in [z]_G \lra \forall n \, g \uhr n \in [z \uhr n]_{G_n}$. \ec

 We define  the subset of $G_n$ of elements that have a preimage in  $[z \uhr t]_{G_t}$: \bc  $U^z_{n,t} = \{ v \in G_n  \colon \ex w \in G_t \, [ w  \in [z \uhr t]_{G_t}  \, \land \, w\uhr n  = v]\}$ \ec 
   Note that  $\bigcap_t U^z_{n,t}$ is the projection of $[z]_G$ into $G_n$.
  
  Note that $[z]_G$ has $C$ as a quotient iff for some $n$, the image  of $[z]_G$ into $G_n$ has $C$ has a quotient. This is equivalent to the condition 
  
  \bc $\ex n \ex s \ge n \fa t \ge s [ U^z _{n,t} \text{ is a subgroup of } G_n \text{ with $C$ as a quotient}]$, \ec which is in $\Sigma^0_2$ form as required.
  \end{proof}

Suppose that $L$ is a profinite group  that  can be described  by its finite quotients among the closed normal subgroups of $G$.   It follows that any  theorem  of the form ``$[z]_G \cong L$ for almost every $z$"   holds for any weakly 2-random $z$.  

\begin{cor} Let $G = \widehat F_\omega$. For each $e$ and each weakly 2-random $z \in G^e$, the topological normal closure $N$ of $z$ is isomorphic to $\widehat F_\omega$.
\end{cor}
To see this, note that by \cite[Th.\ 25.7.3(b)]{Fried.Jarden:06}, $N \cong \widehat F_\omega$ iff each finite group is a quotient of $N$. The $S$-rank function  $r_N(S)$ occuring there, for a finite simple group $S$,  is  defined as follows (see Section 24.9): let $M_G(S)$ be the intersection of all open normal subgroups $X$ of $N$ such that $N/X \cong S$. Then $M_G(S)$ is closed normal, and $G/M_G(S) \cong S^m$ for some cardinal $m$; write $m = r_G(S)$. If $m$ is finite, it is simply   the number of normal open subgroups $X$ of $N$ such that $N/X \cong S$.

Similarly, from \cite[Cor.\ 25.7.6]{Fried.Jarden:06} we obtain a variant when $G = \hat F_m$ has finite rank. Note that any subgroup of finite index is open, and hence free of finite rank; see e.g.\ \cite[Th.3.6.2]{Ribes.Zalesski:00}, which is a profinite version of Schreier's theorem on the rank of finite index subgroups of  discrete free groups of finite rank.
\begin{cor} Let $G = \widehat F_m$ where $m$ is finite. For each $e$ and each weakly 2-random $z \in G^e$, if the topological normal closure   of $z$ has infinite index in $G$, then  it is  isomorphic to $\widehat F_\omega$.
\end{cor}

       \newcommand{\UT}{\ensuremath{\mathrm{UT}_3}}
\newcommand{\Zp}{\ensuremath{\mathbb Z_p}}
\newcommand{\UTp}{\UT(\Zp)}

\section{Describing a profinite structure by a single first-order sentence}

  The following is related to discussions of Nies with M.\ Aschenbrenner and T.\   Scanlon late 2016 at UCLA and UC Berkeley.
They discussed a question that had come up during Nies' visit at Hebrew University earlier that year when talking to Lubotzky and Meiri: can the notion of quasi-finite axiomatisability (QFA, see \cite{Nies:DescribeGroups}) be meaningfully extended to the setting of topological algebra?

 Suppose $\mathcal C$ is  a class  of topological algebraic  structures    in a finite signature $S$. For instance,  $\+ C$ could be the class of profinite separable rings, or the class of profinite separable groups.  \begin{definition} A first-order sentence $\phi$  in $L(S)$ \emph{describes} a structure  $M \in \mathcal C$ if  $M$ is up to topological isomorphism the unique structure in $\+ C$ that satisfies  $\phi$. A structure  $M$ is finitely axiomatisable for $\+ C$ if there is such a $\phi$. 
 \end{definition}
 

We will show that  $\UT(\Zp)$ is finitely axiomatisable within the class of separable profinite groups.
  
  Recall the commutator $[x,y] = xy x^{-1}y^{-1}$. A group $G$ is nilpotent of class 2 (nilpotent-2 for short) if it satisfies the law $[[x,y],z]= 1$. Equivalently $G' \sub C(G)$. This implies distributivity $[uv,w] = [u,w][v,w]$,  and  $[u^n,w] = [u,w]^n$ for any $n \in \ZZ$.  
 
 \subsection{Rings}  
 Throughout let $p$ be a fixed prime number.  Unless otherwise noted,  rings will be commutative and  with 1.
 \begin{thm}  \label{thm: QFA rings}  Let $\+ C$ be the class of profinite rings.
 
 \begin{itemize} \item[(i)] $\ZZ_p$ is finitely axiomatisable for $\+ C$.
 \item[(ii)] $\widehat{{\mathbb Z}}$ is not finitely axiomatisable for $\+ C$.  \end{itemize} \end{thm}
\begin{proof}  

(i) Recall that a local ring $R$ with maximal ideal $m$ is called \emph{Henselian} if Hensel's lemma holds:  if $P$ is a monic polynomial in $R[x]$, then any factorization of its image  in $(R/m)[x]$ into a product of coprime monic polynomials can be lifted to a factorization  of $P$ in $R[x]$. The ring $\mathbb Z_p$  is characterised by saying that it is a Henselian valuation ring with residue field ${\mathbb F}_p$ and a Z-group as valuation group with $p$ generating the maximal ideal.   To express this by a first order sentence, note that a compact valuation ring is complete, so  being Henselian follows from completeness, and we need not include any such axioms.  Likewise, we do not need to describe the value group.   

(ii)  This follows from Feferman-Vaught.   As a ring, $\widehat{{\mathbb Z}} \cong \prod_p {\mathbb Z}_p$.  Then for every sentence $\phi$ in the language of rings there is a finite sequence of sentences $\psi_1, \ldots, \psi_n$ in the language of rings and a formula $\theta(x_1,\ldots,x_n)$ so that for any index set $I$ and any family of rings $R_i$ indexed by $I$ if we set $X_j := \{ i \in I ~:~ R_i \models \psi_j \}$, then \bc $\prod_{i \in I} R_i \models \phi$ if and only if ${\mathcal P}(I) \models \theta(X_1,\ldots,X_n)$.    \ec Consider $\phi$ a supposed QFA formula.   By pigeon hole principle, we can find two distinct primes $\ell \neq q$ so that for all $j \leq n$ we have ${\mathbb Z}_\ell \models \psi_j \Longleftrightarrow {\mathbb Z}_q \models \psi_j$.   Define $R_p := {\mathbb Z}_p$ if $p \neq \ell$ and $R_\ell := {\mathbb Z}_q$.   Then $R := \prod R_p \models \phi$ but $R \not \equiv \widehat{{\mathbb Z}}$ as, for example, $\ell$ is a unit in $R$ but not in $\widehat{{\mathbb Z}}$.
\end{proof}

 \subsection{QFA profinite groups}  
 Given a ring $R$ let $\UT(R)$ denote the set of matrices of the form $A=\left(  \begin{matrix}
   1 & \beta &  \gamma \\
   0 &  1 &  \alpha \\
   0 &   0 &  1   \end{matrix} \right)$ with entries in $R$. For $R= \ZZ$ the standard generators of $\UT(R)$ are $$a = \left(  \begin{matrix}
   1 & 0 &  0 \\
   0 &  1 &  1 \\
   0 &   0 &  1   \end{matrix} \right), b = \left(  \begin{matrix}
   1 & 1 &  0 \\
   0 &  1 &  0 \\
   0 &   0 &  1   \end{matrix} \right).$$ 
   We will  write $q = [a,b]$. Note that $q= \left(  \begin{matrix}
   1 & 0 &  1 \\
   0 &  1 &  0 \\
   0 &   0 &  1   \end{matrix} \right)$.
   For $R= \Zp$ these are topological generators. 
   
 We will show that $\UT(\Zp)$ is finitely axiomatisable within the class of separable profinite groups.
First  we need some preliminaries on $\UT(\Zp)$.
   The \emph{pro-$p$ completion} of a group $G$ is the inverse limit $\varprojlim G/N$ with the canonical projections, where $N$ ranges over normal subgroups of index a power of $p$. If $G$ is f.g.\ nilpotent,  then  we can  let $N$ range  over verbal subgroups $G^{p^s}$, $s \in \NN$ (as they have index a power of $p$). 
   
   We will show that $\UTp$ is the pro-$p$ completion of $\UT(\ZZ)$,  which will imply that it is    the   free pro-$p$ nilpotent-2 group on free generators $a,b$.  
    
   Let now $R= \ZZ$. Every matrix $A$ as above can be uniquely written as $A = (\alpha, \beta, \gamma) = a^\alpha b^\beta q^\gamma$ where $\alpha, \beta, \gamma  \in R$. We have  
   \begin{equation} \label{eqn:1}  (\alpha, \beta, \gamma) (\alpha', \beta', \gamma') = (\alpha  + \alpha', \beta + \beta', \gamma + \gamma' + \alpha'\beta)\end{equation}
   and for any $r \in \ZZ$,   
 \begin{equation} \label{eqn:2}   (\alpha, \beta, \gamma) ^r =  (r\alpha, r  \beta, r (\gamma+ \alpha\beta)\end{equation}
 
 The following is well-known; see e.g.\ \cite{Karg.Merz:79}.
 \begin{fact} \label{fact: KM exercise} $\UT(\ZZ)$ is the free abstract nilpotent-2 group on free  generators $a, b$.  \end{fact}
\begin{proof} Suppose $G$ is a nilpotent-2 group generated by $u,v$.  Let $w = [u,v]$. Every element of $G$ can be expressed (not necessarily uniquely) in the form $u^\alpha v^\beta w^\gamma$. Since (\ref{eqn:1}) applies also in $G$,  $a \mapsto u, b \mapsto v$ extends to a group homomorphism.  \end{proof}


 \begin{fact} \label{fact:compl} $\UTp$ is the pro-$p$ completion of $\UT(\ZZ)$.   \end{fact}
\begin{proof} In the setting of topological  rings,   $\Zp = \varprojlim_s \ZZ / p^s \ZZ$. Therefore \bc  $\UTp = \varprojlim_s  \UT(\ZZ / p^s \ZZ)$. \ec Write $G= \UT(\ZZ)$. Let $ \UT(p^s \ZZ)$  denote the normal subgroup of $G$ consisting of matrices with entries off the main diagonal divisible by $p^s$. Then $|G :  \UT(p^s \ZZ)| = p^{3s}$, so it suffices to show  that  for the verbal subgroups $G_s = G^{p^s}$  we have  \bc $G_s  \ge \UT(p^s \ZZ)$. \ec   To this end, suppose we are given $(\alpha, \beta, \gamma) \in \UT(p^s \ZZ)$. Let $\alpha' = p^{-s} \alpha, \beta ' = p^{-s} \beta$ and  $\gamma' = p^{-s} \gamma - \alpha'\beta'$. Then $(\alpha', \beta', \gamma')^{p^s} = (\alpha, \beta, \gamma)$  by (\ref{eqn:2}). 
\end{proof}

 As a consequence, $\UTp$ can be seen as a $\Zp$-module: for $x \in \Zp$ and $g \in \UTp$ define $g^x = \lim_n g^{x\, \uhr n}$ where $x \uhr n$ denotes the last $n$ digits of $x$.  This limit exists because $p^s \mid x$ implies the projection of $g^{x\, \uhr n}$ 
 in $ \UT(\ZZ / p^s \ZZ)$ vanishes. 
 
 More generally, let $\widehat \ZZ$ denote the free profinite group of rank $1$. For any profinite group $G$ $g \in G$ and $\lambda \in \widehat \ZZ$ one can define exponentiation $g^\lambda$ as $\phi(\lambda)$ where $\phi \colon \widehat \ZZ \to G$ is the unique homomorphism with $\phi(1) = g$.  The usual laws of exponentiation hold. See \cite[Section 4.1]{Ribes.Zalesski:00}.
 
  The following fact should be  well-known,  but is somewhat  hard to find in the literature.
 
 \begin{prop} \label{free} $\UTp$ is the  pro-$p$ nilpotent-2 group on free generators $a,b$. \end{prop}
 \begin{proof} There are two ways to see this.

 (1) Suppose that we are given  a nilpotent-2 pro-$p$ group $H$  topologically generated by $u,v$. Each subgroup of index  a power of $p$ is open (Serre). So $H$ is its own pro-$p$ completion. 
 
 Let $\Theta \colon \UT(\ZZ) \to H$ be the abstract group homomorphism given by $a \to u,  b \to v$ where $a,b$ are seen as standard generators of $\UT(\ZZ)$. Let $\widetilde U$ be the  completion of $\UT(\ZZ)$ with respect to the (restricted) system $\Theta^{-1}(H^{p^s})$, $s\in \NN$, where $H^{p^s}$ is the verbal subgroup  as above. Then there are natural continuous epimorphisms $\UTp \to \widetilde U$ by Fact~\ref{fact:compl} and $\widetilde U \to H$ since $H$ is its own pro-$p$ completion. Their composition maps $a $ to $u$ and $b$ tp $v$. 
 
 (2) Every matrix $A \in \UTp$   can be uniquely written as $A =  a^\alpha b^\beta q^\gamma$ where $\alpha, \beta, \gamma  \in \Zp$. Now argue as in Fact~\ref{fact: KM exercise}. (This second argument requires verification of some   facts on exponentation in pro-$p$ groups.) 
 \end{proof}

   \begin{thm} \label{thm:UTQFA} $\UT(\Zp)$ is finitely axiomatisable within the class of separable profinite groups.
 
 In fact there is a first-order formula $\phi(r,s)$ in the language of groups such that for each separable profinite group $G$, if $G \models \phi(c,d)$ for $c,d \in G$, then $a \mapsto c, b \mapsto d$ yields a topological isomorphism $\UT(\Zp) \cong G$.  \end{thm}

 \begin{proof} 
 We follow the general outline of \cite[Thm.\ 5.1]{Nies:03}, where it is shown that $\UT(\ZZ)$ is QFA within the class of f.g.\ abstract groups. As explained there in more detail,  for any ring $R$   the Mal'cev formula $\mu(x,y,z; r,s)$ defines the ring operation  $M_{r,s}$ on the centre $C(\UT(R)) \cong (R,+)$ when $r,s$ are assigned to the standard generators $a,b$ (also see \cite{Nies.Thomas:08})s.  
 
Sentence  $\alpha_1$ expresses of a profinite group $G$ that $G$  is nilpotent-2, and the centre $C = C(G)$ equals the set of commutators (in particular, $G_\mathrm{ab} = G/C$). Since $C$ is closed, by Theorem~\ref{thm: QFA rings}, there is a   formula $\gamma(r,s)$ expressing that $(C, + , M_{r,s})$ is isomorphic to $\Zp$; in addition, $\gamma$ expresses that $[r,s]$ is the neutral element $1$ of this  ring. Finally, a sentence  $\alpha_3$ expresses that $pG/C$ has index $p^2$ in $G/C$.   Let $\phi(r,s) \equiv  \alpha_1 \land \gamma(r,s) \land \alpha_3$. 

Suppose now that $G \models \phi(c,d)$. 
\begin{claim}  $G_\mathrm{ab}\cong \Zp \times \Zp$. \end{claim}
 As in \cite[Thm.\ 5.1]{Nies:03} since the centre $C$ is torsion free, $G_\mathrm{ab}$ is torsion free: if $u \not \in C$ then $[u,v] \neq 1$ for some $v \in G - \{1\}$. Then $[u^n,v] = [u,v]^n \neq 1$ so that $u^n \not \in C$. 
 
 Next, since  $G_\mathrm{ab}$ is profinite, by the structure theorem (e.g. \cite[Thm. 4.3.8]{Ribes.Zalesski:00})  $G_\mathrm{ab} = \prod_q \mathbb Z_q  ^ {m(q)}$ where $q $ ranges over the primes and $m(q)$ is a cardinal.  Then $m(q) = 0$ for $q \neq p$. For otherwise we can  take $ v \in G - C$ such that in  $G_\mathrm{ab}$ we have $ p^n \mid vC$ for each $n$. Choose $z \in G$ such that $ [v,z]  \neq 1$, and take $w \in G$ such that $p^n wC = vC$. Then $k = [w^{p^n}, z] = [w,z]^{p^n}$ so that $p^n \mid k$ for each $n$,  contrary to the fact that $C \cong \Zp$ as an abelian group.  
Since $G \models \alpha_3$ we have $m(p) =2$. This shows the claim.

 Further, $G$ is a pro-$p$ group since the class of such groups is closed under extensions \cite[Thm.\ 2.2.1(e)]{Ribes.Zalesski:00}.

For a topological group $G$ and $S \sub G$ let  $\la S \ra$ denote the closure of the subgroup   generated by  $S$. 
\begin{claim} $G = \la c,d \ra$. \end{claim} 
\begin{claimproof} Since $G \models \alpha_3$ we have $\la [c,d]\ra  = C$. So it suffices to show that $\la Cc, Cd \ra = G_{\mathrm{ab}}$. Pick $g,h \in G$ such that $\la Cg, Ch \ra = G_{\mathrm{ab}}$.  There are $x,y,z,w \in \Zp$ and $u,v \in C$  such that $c = ug^x h^y$ and $d = v g^z h^w$. Then $[c,d] = [g,h]^{xy-zw}$.  On the other hand $[c,d]^k = [g,h]$ for some $k \in \Zp$. So the determinant $xy-zw$ is a unit in $\Zp$, whence $\la Cc, Cd \ra = G_{\mathrm{ab}}$. 
\end{claimproof}

By Prop.\ \ref{free} there is a continuous group homomorphism $\Theta \colon \UTp \to G$ such that $\Theta(a) =  c$ and $\Theta(b) = d$. Then $\Theta([a,b]) = [c,d]$ so $\Theta$ induces an isomorphism $C(\UTp) \to C(G)$. Also $\Theta$ induces an isomorphism 
 $\UTp_{\mathrm{ab}} \to G_{\mathrm{ab}}$. Hence $\Theta$ is an isomorphism as required.
 \end{proof}

  \part{Metric spaces and descriptive set theory}
  
  \newcommand{\rank}{\text{rank}}

\section{Turetsky and Nies: Scott rank of Polish metric spaces - a computability approach}
We give a proof based on computability theory  of Doucha's result  \cite{Doucha:14} that the Scott rank of Polish metric spaces $M$ is at most $\omega_1 +1$.  We prove that the Scott rank of each pair of tuples $\ol a, \ol b$ of the same length is bounded  by $\omega_1^{\ol a, \ol b, M}$. Here we view metric spaces as structures in a countable language with distance relations $R_q (x,y)$, where $q$ is a positive rational,  intended to express that the distance of $x,y$ is less than $q$.  William Chan announced this in 2016, giving a proof involving admissible sets. 

For a structure $M$, $\ol a, \ol b \in M^n$, and a linear order $L$,   the Ehrenfeucht-Fra\"iss\'e game $G^L_M(\ol a,\ol b)$+  is played as follows:
\begin{itemize}
\item On the $i$th round, Player 1 chooses a $z_i \in L$ with $z_i <_L z_{i-1}$ when $i > 0$, and either chooses an element $a_{n+i} \in M$ or an element $b_{n+i} \in M$.
\item Player 2 then chooses whichever of of $a_{n+i}$ or $b_{n+i}$ Player 1 did not.
\end{itemize}
After round $i$, if the map from $a_0a_1\dots a_{n+i}$ to $b_0b_1\dots b_{n+i}$ is not a partial isomorphism, then Player 1 wins.  The game ends in a win for Player 2 after either $\omega$ many rounds, or if Player 1 cannot choose a $z_{i+1} <_L z_i$, and Player 1 has not already won.

We extend these games to metric spaces, replacing partial isomorphism with partial isometry.

\begin{fact}
If $M$ is a countable structure or a Polish  metric space and $L$ is ill-founded, then Player 2 has a winning strategy in $G^{L}_M(\ol a,\ol b)$ iff there is an automorphism or autoisometry of $M$ taking $\ol a$ to $\ol b$.
\end{fact}

\begin{definition}
For a structure or metric space $M$, define $\rank^M(\ol a,\ol b)$ to be the least ordinal $\alpha$ for which Player 2 does not have a winning strategy in $G_M^\alpha(\ol a,\ol b)$, or $\rank^M(\ol a,\ol b) = \infty$ if there is no such $\alpha$.

Define $\rank^M(\ol a) = \sup\{ \rank^M(\ol a,\ol b) : \rank^M(\ol a,\ol b) < \infty\}$.

Define $\rank(M) = \sup\{ \rank^M(\ol a) + 1 : \ol a \in M\}$.
\end{definition}

Note that in some versions e.g. Doucha's, the ``$+1$" is omitted.

\begin{fact}
For a computable structure or Polish space $M$ and any reasonable definition of Scott Rank, $SR(M) \le \rank(M)$.
\end{fact}

\begin{definition}
If $M$ is a Polish metric space and $D \subseteq M$ is dense, define $G^L_M(\ol a,\ol b,D)$ exactly as $G^L_M(\ol a,\ol b)$, save that Player 1's choice of elements is restricted to $D$.  Player 2 is still allowed to choose any element from $M$.

Define $\rank^M_D$ and $\rank(M,D)$ as above, using $G^\alpha_M(\ol a,\ol b,D)$ in place of $G^\alpha_M(\ol a,\ol b)$.
\end{definition}

\begin{remark}
If $L$ is countable, a strategy for Player 2 is coded by a real, given some numbering of $L$ and $D$.  Player 1's possible plays at any give round are each coded by an element of $\omega$, while Player 2's responses are given by Cauchy sequences from $D$.
\end{remark}

\begin{remark}
Given numberings of $D$ and $L$ and a real, checking that this real codes a winning strategy for Player 2 is arithmetical relative to the metric on $D$.  We must check that every response is a Cauchy sequence, and that every partial play of the game results in a partial isometry.
\end{remark}

\begin{remark}
Any winning strategy of Player 2's for $G^L_M(\ol a,\ol b)$ restricts to a winning strategy for $G^L_M(\ol a,\ol b,D)$, and so $\rank^M(\ol a,\ol b) \le \rank^M_D(\ol a,\ol b)$.
\end{remark}

\begin{fact}
If $M$ is a Polish  metric space, $D \subseteq M$ is dense and $L$ is ill-founded, then Player 2 has a winning strategy in $G^{L}_M(\ol a,\ol b,D)$ iff there is an autoisometry of $M$ taking $\ol a$ to $\ol b$.
\end{fact}

\begin{thm}
If $M$ is a Polish  metric space, $D \subseteq M$ is dense and $\rank^M_D(\ol a,\ol b) \ge \omega_1^{\ol a,\ol b,M\uhr{D}}$, then there is an autoisometry of $M$ taking $\ol a$ to $\ol b$, and so $\rank^M(\ol a,\ol b) = \infty$.

Thus  $\rank^M(\ol a,\ol b) <  \omega_1^{\ol a,\ol b,M\uhr{D}}$  unless $\ol a, \ol b$ are autoisometric, and therefore $\rank^M(\ol a) \le \omega_1$ and $\rank(M) \le \omega_1+1$.
\end{thm}

\begin{proof}
Suppose $\rank^M_D(\ol a,\ol b) \ge \omega_1^{\ol a,\ol b,M\uhr{D}}$.  Consider the formula $\theta(e)$ stating that $\Phi_e^{\ol a,\ol b,M\uhr{D}}$ gives a total linear order $L$, and there is a real which codes a winning strategy for Player 2 in $G^L_M(\ol a,\ol b,D)$.  Note that $\theta$ is $\Sigma^1_1(\ol a,\ol b,M\uhr{D})$.  By assumption, $\theta(e)$ holds for every $e$ with $\Phi_e^{\ol a,\ol b,M\uhr{D}}$ well-ordered, and so by $\Sigma^1_1$-bounding it must hold for some $e$ with $\Phi_e^{\ol a,\ol b,M\uhr{D}}$ ill-founded.  As observed before, this means that there is an autoisometry of $M$ taking $\ol a$ to $\ol b$.
\end{proof}
We strengthen the  result of William Chan  that a  rigid Polish metric space has countable Scott rank.
\begin{prop} Let $M$ be a Polish metric space such that the isometry relation on tuples of the same length is $\Delta^1_1$. Then the Scott rank of $M$ is computable in $M\uhr  D$. \end{prop}
\begin{proof}  By hypothesis the following property of $e \in \omega$  is $\Sigma^1_1$:    $\Phi_e ^{M\uhr D}$ codes a linear order $L$ such that \[\exists n \exists \ol a, \ol b \in M^n  \, [ \ol a \not \approx \ol b  \ \land \ \text{Player 2 has a winning strategy in } G^L_M(\ol a, \ol b, D).] \]

By the argument above for each such $e$,   $\Phi_e ^{M\uhr D}$  is a well-ordering. By $\Sigma^1_1$ bounding, the set of such $\Phi_e ^{M\uhr D}$  is then bounded by an ordinal computable in $M \uhr D$.
\end{proof}

   
 \part{Model theory and definability}

   \section{Some open questions on computability and structure} 
Noam Greenberg, Alexander Melnikov, Andr\'e Nies and and Dan Turetsky worked  at the Research Centre Coromandel April 18-22. They discussed the following open questions.

\begin{question} Let $A$ be a computable $\omega$-categorical structure in a finite signature. Show that  $S_A$, the set of indices for computable structures isomorphic to $A$,   is arithmetical.
\end{question}

  Turetsky showed that $S_A$ can be arbitrarily high in the arithmetical hierarchy depending on the arity of the language. According to Melnikov,  an affirmative  answer follows from Uri Andrew's thesis around p.\ 42, related to the Hrushovski construction.

\begin{question} Is every  computable closed subgroup $G $ of $S_\infty$  topologically isomorphic to the automorphism group of a computable structure? \end{question} 
Such a group is given as a $\PI 1$ class of pairs $f,g$ of functions in Baire space such that $g =f^{-1}$. They can be seen as pruned subtrees of     the  tree $\mathbb T $  of all pairs $\la \sss, \sss' \ra$ of strings of the same length $n$ such that
$\sss(i) = k \lra \sss'(k) = i$ for each $i,k < n$.  It is well known that $G$ is topologically isomorphic to the automorphism group of a some countable  structure, namely the one consisting of all the $n$-orbits, seen as named $n$-ary relations, for each $n$. However, the orbit relation is only computable in $\mathcal O$ in general, so this structure is merely computable in $\+ O$. 

\begin{question} Does computably categorical imply relatively $\Delta^1_1$-categorical? \end{question}

\begin{question} Let $G$ be a f.g.\ group with $\PI 1$ word problem. Is $G$ embeddable into the group of computable isometries of a computable metric space?  
\end{question}
Morozov \cite{Morozov:00} showed that one cannot always choose the space discrete; I.e., there is an example of $G$ that is not a subgroup of the computable permutations of $\NN$.  Yet, one can show that this particular example can be realised as a group of computable isometries.

Further suggestions for study (Melnikov): \bi \item  the partial order of primitive recursive presentations of a structure (such as $(\QQ, <)$ under the preordering of isomorphisms that are p.r.\ (without the inverse necessarily being p.r.) For instance, do you get the same degree structure for $(\QQ, <)$ and the countably atomless Boolean algebra?  

\item  Does Markov computable for compact groups imply fully computable? 
\ei

\def\cprime{$'$} \def\cprime{$'$}



\end{document}